\documentclass{amsart}
\usepackage{amssymb}

\newtheorem{theorem}{Theorem}[section]
\newtheorem{lemma}[theorem]{Lemma}
\newtheorem{proposition}[theorem]{Proposition}

\theoremstyle{definition}
\newtheorem{example}[theorem]{Example}

\numberwithin{equation}{section}

\newcommand{\Q}{{\mathbb{Q}}}
\newcommand{\C}{{\mathbb{C}}}
\newcommand{\Z}{{\mathbb{Z}}}
\newcommand{\Tr}{\mathop{\mathrm{Tr}}}
\newcommand{\OL}{\mathcal{O}}

\begin{document}

\title[Computing normal integral bases]{Computing normal integral bases 
of abelian number fields}

\author{Vincenzo Acciaro}
\address{Dipartimento di Economia\\ Universit\`a di Chieti--Pescara\\
Viale Pindaro 42\\ I--65126 Pescara, Italy}



\date{}

\begin{abstract}
Let $L$ be an abelian number field of degree $n$ with Galois group 
$G$. 
In this paper we study how to compute efficiently a normal integral basis 
for $L$, if there is at least one, assuming that the group $G$ and 
an integral basis for $L$ are known.
\end{abstract}

\maketitle

\section{Introduction}\label{sec-intro}

\noindent 
Let $L$ be a Galois number field and let $\OL$ denote its ring of 
algebraic integers. A major problem of algorithmic algebraic number theory 
is to compute efficiently a normal integral basis for $L$ over $\Q$, 
that is a basis of $\OL$ as a $\Z$--module which is made up of all the
conjugates of a single algebraic integer. 

Although the theoretical aspects of the question are well understood, 
the effective construction of normal integral basis has been accomplished 
only in few particular cases.

In this paper we focus our attention to the abelian case. We extend the results obtained 
in \cite{acciarocangelmi} for abelian number fields of exponent 2, 3, 4 and 6 to any  abelian number field.

From a theoretical point of view, the Hilbert-Speiser theorem
asserts that an abelian number field $L$ admits a normal integral basis 
if and only if the conductor of $L$ is squarefree.
When this is the case, 
such a basis may be constructed by means of Gaussian periods:
if we denote the conductor of $L$ by $f$, then 
$\Tr_{\Q(\zeta_f)/L}(\zeta_f)$ generates a normal integral basis for $L$ 
(see \cite{narkiewicz}). 
From a practical point of view, this is quite unsatisfacory, since we
need to work in $\Q(\zeta_f)$ whose degree $\varphi(f)$
may be very large in comparison to $[L:\Q]$. 
Hence the computation of the relative trace might be very expensive.

The algorithm presented here has the advantage that it is
essentially independent of the conductor. 

In Section \ref{sec-gen}, we fix some notation, give some basic results, 
and reduce our problem to that of computing a generator of a certain ideal 
of the group ring $\Z[G]$ of the Galois group $G$ of $L$. 
In Section \ref{sec-dec}, we construct an explicit homomorphic embedding 
of this group ring into the finite product of some cyclotomic rings. 
In Section \ref{sec-algo}, we describe the essential steps of our 
algorithm. 
Finally, in Section \ref{sec-exmpl}, we give some explicit numerical 
examples.

\section{Notation and general framework}\label{sec-gen}

From now on, let $L$ be an abelian number field of degree $n$, 
let $\OL$ be its ring of algebraic integers, and let $\alpha$ be a 
primitive element for the extension $L/\Q$, so that $L=\Q[\alpha]$. 
We can assume that $\alpha \in \OL$, and 
we denote the minimal polynomial of $\alpha$ over $\Q$ by $m(x)$, which is 
therefore a monic polynomial of degree $n$ with integer coefficients. 

Let $G$ be the Galois group of $L/\Q$. Under the Extended Riemann Hypothesis
it is possible to compute efficiently $G$ by using the algorithm described 
in \cite{acciaro-gal}. Such algorithm gives us explicitly the action of the 
elements of $G$ on $\alpha$. 
So, we let $G = \left\{ g_1, \ldots, g_n \right\}$, and 
put $\alpha_i = g_i(\alpha)$, for $i=1,\ldots,n$. 
Then, we can assume that the conjugates of $\alpha$ constitute a basis 
of $L$ as a vector space over $\Q$.

Finally, we let $\left\{ \beta_1, \ldots, \beta_n \right\}$ be an integral 
basis of $\OL$, that is a basis of $\OL$ as a $\Z$--module. 
We recall that an integral basis can be computed using the algorithms 
described in \cite{cohen} and \cite{zassenhaus}.

Let $\Z[G]$ and $\Q[G]$ denote respectively the group ring of $G$ 
over $\Z$ and over $\Q$. 
The action of $G$ on $L$ can be extended by linearity to 
an action of $\Q[G]$ (or $\Z[G]$), setting 
\[
( \sum_{g \in G} a_g g) x = \sum_{g \in G} a_g g(x) ,
\]
for all $a_g$ in $\Q$ (resp.\ $\Z$), and all $x\in L$.

We say that $L$, or $\OL$, has a \emph{normal integral basis} when there 
exists $\theta \in \OL$ such that the conjugates of $\theta$ constitute a 
basis of $\OL$ as a $\Z$--module. In such a case we call $\theta$ a 
\emph{normal integral basis generator}.

If $\theta \in \OL$, then $\theta$ is a normal integral basis generator if and only if the discriminant of the set 
$ \left\{ g_1(\theta),\ldots, g_n(\theta) \right\} $
equals the discriminant of $L$. 
We will need to perform such a check just once, and 
our algorithm will return $\theta$ if such a $\theta$ exists,
and \lq$\theta$ does not exist\rq\ otherwise.

The fact that the conjugates of $\alpha$ constitute a basis for $L/\Q$ can 
be rephrased by saying that $L$ is free of rank one as a $\Q[G]$--module.
The integer counterpart of this property is given next.
\begin{lemma}
\label{th-nib}
The field $L$ possesses a normal integral basis if and only if
the ring $\OL$ is free of rank one as a $\Z[G]$--module.
\end{lemma}

Assume now that $\OL$ has a normal integral basis and that $\theta$ is a 
normal integral basis generator, so that $\OL = \Z[G]\theta$.

Since $ \OL = \sum_{i=1}^n \Z\beta_i $
and $L$ is normal, we also have $ \OL = \sum_{i=1}^n \Z[G]\beta_i $
(where this sum is not direct).
In other words, the $\beta_i$'s form a set of generators of
$\OL$ as a $\Z[G]$--module. 
We would like to compute a single free generator of $\OL$ 
from the given set $\{\beta_1,\ldots,\beta_n \}$.

Let $D \in \Z$ be such that 
\[
\OL \subseteq \Z[G]\frac{\alpha}{D} .
\]
For instance, we could take $D$ equal to the discriminant 
of the set $\{\alpha_1,\ldots,\alpha_n\}$. 
For the sake of convenience, we put $\alpha'=\alpha/D$, and 
$\alpha'_i = g_i(\alpha') = \alpha_i/D$, for $i=1,\ldots,n$. 

On the one hand, we have $ \theta \in \Z[G]\alpha' $, so that 
$\theta = \sum_{i=1}^n t_i \alpha'_i$, for some $t_i \in \Z$. 
In other words, if we let $t = \sum_{i=1}^n t_i g_i$, then we have 
$\theta = t \alpha'$, where $t \in \Z[G]$.

On the other hand, we can express each element $\beta_j$ of the known 
integral basis as a linear combination with integral coefficients 
of the elements 
$ \alpha'_1, \ldots, \alpha'_n $.
In other words, for $ j=1,\ldots,n $ we can write
\begin{equation}\label{eq-bij}
\beta_j = \sum_{i=1}^n b_{ij} \alpha'_i , 
\end{equation}
with $b_{ij} \in \Z$. This is equivalent to say that 
$\beta_j = b_j\alpha'$,
where 
\begin{equation}\label{eq-bi}
b_j = \sum_{i=1}^n b_{ij}g_i \in \Z[G] .
\end{equation}

Therefore we have 
\[
\OL = \Z[G]t\alpha' = \left( \sum_{j=1}^n \Z[G]b_j \right)\alpha' ,
\]
and, since $\alpha$ gives a normal basis for $L/\Q$ and 
the same is true for $\alpha'$, 
\[
\Z[G]t = \sum_{j=1}^n \Z[G]b_j .
\]
In conclusion, we have reduced our problem to the problem of finding 
a generator of the ideal of $\Z[G]$ generated by the set 
$\{b_1,\ldots,b_n\}$. For future reference, let us call this ideal $I$, 
that is let us define
\begin{equation}\label{eq-ideal}
I = \sum_{j=1}^n \Z[G] b_j .
\end{equation}
We complete our arguments and state the main results of this section 
in the following theorem.

\begin{theorem}\label{th-nibg}
Let $L$ be an abelian number field, and let $I$ be the ideal of $\Z[G]$ 
defined by (\ref{eq-ideal}). 
Then, $L$ has a normal integral basis if and only if $I$ is principal.
More precisely, if $\alpha' \in L$ is such that $L=\Q[G]\alpha'$ and 
$\OL \subseteq \Z[G]\alpha'$, then we have:
\begin{itemize}
\item 
If $\theta$ is a normal integral basis generator, 
and $\theta = t\alpha'$, with 
$t \in \Z[G]$, then $I=\Z[G]t$. 
\item 
If $I$ is principal and $t \in \Z[G]$ is a generator of $I$, 
then $t\alpha'$ is normal integral basis generator.
\end{itemize}
\end{theorem}
\begin{proof}
The above arguments show that if $L$ has a normal integral basis and 
$\theta$ is a normal integral basis generator then $I$ is principal and 
is generated by $t$, which is defined by the relation $\theta=t\alpha'$. 
Then, it is easily seen that the converse holds true, and that 
if $t$ is a generator of $I$, then the element $t\alpha'$ is integral 
and it is a normal integral basis generator.
\end{proof}

\section{Decomposition of the group ring $\Q G$ }\label{sec-dec}
We begin by stating a result about rational group rings of finite abelian
groups, proved by S. Perlis and G. Walker in \cite{perliswalker}
as well as by Higman \cite{higman}.
\begin{theorem}[Perlis, Walker]\label{th-perliswalker}
Let $G$ be a finite abelian group of order $n$. Then
\[
\Q G \cong \oplus_{q|n}     {a_q}        \,\Q(\zeta_q) 
\]
where 
$ {a_q}      \Q(\zeta_q)   $ denotes the direct sum 
of $a_q$ copies of $\Q(\zeta_q)$.
Moreover, $a_q = n_q/\varphi(q)$, 
with $n_q$ equal to the number of elements of order $q$ in $G$.
\end{theorem}

A very enjoyable proof of this result can be found in \cite{polcino}.
However, for our purposes, we are going to construct this isomorphism explicitly, 
and then we will restrict it to $\Z G$.

Following Higman's approach
the idempotent $e_\chi \in \C G$ 
associated to any character (representation of degree 1) $\chi$ is defined
as:
\[
e_\chi = \frac{1}{|G|} \,\sum_{g \in G} \chi(g^{-1}) \, g .
\]
Let $\chi_1, \ldots, \chi_n$ be the irreducible characters of $G$, 
and denote the character group by $ \widehat{G}$.
From now on, we write $e_i$ in place of $e_{\chi_i}$, for short. 

Extending the characters linearly to the full group ring, we see that
any 
element $h$ of $\C G$ admits two unique representations:
\[
h = \sum_{i=1}^n h_i g_i
=
\sum_{i=1}^n c_i e_i ,
\]
and we have the well known decomposition 
\[
\C G = \C e_1 \oplus \cdots \oplus \C e_n .
\]
If we denote the character matrix 
$\left( \chi_i(g_j) \right)_{i,j=1,\ldots,n}$ of $G$ by $A$, 
then we have the following equality:
\begin{equation}\label{eq-ci} 
\left( \begin{array}{c} c_1 \\ \vdots \\ c_n \end{array} \right) 
 = A 
\left( \begin{array}{c} h_1 \\ \vdots \\ h_n \end{array} \right) . 
\end{equation}
Conversely, as $A$ is not singular, 
if the coefficients $(c_i)$ of an element $h$ of $\C G$ are known,
we can easily compute the coefficients $(h_i)$:
\begin{equation}\label{eq-hi}
\left( \begin{array}{c} h_1 \\ \vdots \\ h_n \end{array} \right) 
= A^{-1} 
\left( \begin{array}{c} c_1 \\ \vdots \\ c_n \end{array} \right) . 
\end{equation}
Here we can compute directly the inverse of the matrix $A$, 
as it is well known that 
$ (A^{-1})_{ij}=\overline{\chi_j(g_i)}/|G|= {\chi_j(g_i^{-1})}/|G| $.
Hence, we have:
\begin{proposition}
The map $\C G \rightarrow \C^n$ which sends $h$
to $\left(\chi_1(h),\ldots,\chi_n(h)\right)$ is a $\C$-algebra isomorphism, 
and the associated matrix with respect to the bases 
$\left(g_1,\ldots,g_n\right)$ and $\left(e_1,\ldots,e_n\right)$ 
is the character matrix $A$. 
\end{proposition}

Now, we turn our attention back to the group ring $\Q G$ and to 
Theorem \ref{th-perliswalker}. 
For each irreducible character $\chi_i$, $i=1,\ldots,n$, let us denote 
its order in $\widehat G$ by $q_i$. 
Therefore $\chi_i(g)$ is a root of unity of order $q_i$ 
for all $g \in G$, and is a primitive root of unity of 
order $q_i$ for some $g \in G$.

Note that if $\chi$ is a character of $G$ with values in some 
Galois number field $K$ (actually, some $\Q(\zeta)$, for some root 
of unity $\zeta$) and $\sigma$ is an automorphism of $K$, 
then $\sigma(\chi)$ is another character of $G$ 
with values in the same field $K$. 
We say that $\chi_i$ and $\chi_j$, with values in some $K$, 
are \emph{algebraically conjugate} 
if there exists an automorphism $\sigma$ of $K$ such that 
$\chi_j = \sigma(\chi_i)$.
Thus,  $\chi_i$ and $\chi_j$  are algebraically conjugate
if and only if they are equivalent over $\Q$.
However, for our purposes, it is better to think in terms of
conjugation rather than in terms of equivalence of the associated representations.

Let us select from the set $\left\{\chi_1, \ldots, \chi_n\right\}$ a 
maximal subset of irreducible characters such that no one of them is 
algebraically conjugate to some other.

Let us assume that $\chi_{r}$ is an irreducible character
of $G$ of order $o_r$, so that
$\chi_{r}$ maps $\Q G$ onto $\Q[\zeta_{o_r}]$.
If $\sigma$ is an automorphism of $\Q[\zeta_{o_r}]$, then $\sigma$ sends $\zeta_{o_r}$ to
$\zeta_{o_r}^{s}$, where $o_r$ and $s$ are coprime, and
therefore $\chi_r$ to $\chi_r^s$.

Viceversa, if $\chi_a$ and $\chi_r$ are two characters of $G$
such that $\chi_a = \chi_r^s$ in  $\widehat{G}$,
with $(r,o_r)=1$, then $\chi_a = \sigma(\chi_r)$ where  $\sigma$ sends $\zeta_{o_r}$ to
$\zeta_{o_r}^{s}$ in $\Q[\zeta_{o_r}]$.

Thus, the computation of a maximal subset of irreducible characters such that
no one of them is algebraically conjugate to some other  can be done by
applying one of the following methods:
\begin{itemize}
\item
For each character $\chi$, remove all co-prime powers of $\chi$ from the  list;
\item
For each character $\chi$, form the orbit of $\chi$ under the action of the
Galois group of $\Q[\zeta_{o_r}]$, and remove all the conjugates of $\chi$ from
list.
\end{itemize}

After rearranging the original set, we can assume that 
$\left\{\chi_1, \ldots, \chi_k\right\}$ is such a subset, 
where $1\leq k\leq n$. 

For further reference, for any character which has not been selected, 
we keep track of the selected character which is algebraically conjugate 
to it, and of the automorphism which gives such a relation. 
Precisely, for any $i$, where $k<i\leq n$, we take note of the (unique) index 
$k_i$, with $1\leq k_i \leq k$, such that $\chi_i$ and $\chi_{k_i}$ are 
algebraically conjugate, and also of the automorphism $\sigma_i$ 
such that $\sigma_i(\chi_{k_i})=\chi_i$.

It is easy to see that, for any divisor $q$ of $n$, the number of 
irreducible characters of order $q$ equals $\varphi(q)$ times the 
number of subgroups of $G$ of order $q$, and the maximal number of 
irreducible characters of order $q$ which are not algebraically conjugate 
equals the number of subgroups of $G$ of order $q$. 
In this way, we recover the same arithmetic conditions stated in 
Theorem \ref{th-perliswalker}. 

Summarizing, we have:
\begin{proposition}\label{th-qg}
Assume that $\{\chi_1, \ldots ,\chi_k\}$ is a maximal subset of irreducible 
characters which are not pairwise algebraically conjugate. 
Then, the map $\phi : \Q G \rightarrow 
\Q(\zeta_{q_1})  \oplus \cdots   \oplus \Q(\zeta_{q_k})$ 
which sends $h$ to $\left(\chi_1(h),\ldots,\chi_k(h)\right)$
is a $\Q$-algebra isomorphism. 
Moreover, if we let $B$ be the matrix made of the first $k$ rows of the 
character matrix $A$, then for $h = \sum_{i=1}^n h_i g_i$, we have 
$\phi(h) = (c_1,\ldots,c_k)$, where 
\begin{equation}\label{eq-ci-b}
\left( \begin{array}{c} c_1 \\ \vdots \\ c_k \end{array} \right) 
= B 
\left( \begin{array}{c} h_1 \\ \vdots \\ h_n \end{array} \right) . 
\end{equation}
Finally, if $(c_1,\ldots,c_k) \in 
\Q(\zeta_{q_1})  \oplus \cdots   \oplus \Q(\zeta_{q_k})$, 
we let $c_i=\sigma_i(c_{k_i})$ for all $i$ with $k<i\leq n$, 
and we have $\phi^{-1}((c_1,\ldots,c_k)) = \sum_{i=1}^n h_i g_i$, where 
\begin{equation}\label{eq-hi-b}
\left( \begin{array}{c} h_1 \\ \vdots \\ h_n \end{array} \right) 
= A^{-1} 
\left( \begin{array}{c} c_1 \\ \vdots \\ c_n \end{array} \right) . 
\end{equation}
\end{proposition}

At last, we consider the integral case, by restricting the map 
defined on $\Q G$, and we obtain:

\begin{proposition}\label{th-zg}
The restriction of $\phi$ to $\Z G$ gives a $\Z$--algebra homomorphism 
\[
\psi: \Z G \hookrightarrow 
\Z[\zeta_{q_1}]   \oplus \cdots   \oplus \Z[\zeta_{q_k}] ,
\]
which is injective. 
If $h = \sum_{i=1}^n h_i g_i$, then $\psi(h) = (c_1,\ldots,c_k)$, where 
$(c_1,\ldots,c_k)$ is given by (\ref{eq-ci-b}). 
If $(c_1,\ldots,c_k) \in \psi(\Z G)$, 
then $\psi^{-1}((c_1,\ldots,c_k)) = \sum_{i=1}^n h_i g_i$, 
where $(h_1,\ldots, h_n)$ is given by (\ref{eq-hi-b}). 
\end{proposition}
\begin{proof}
If $\chi_j$ is an irreducible character of $G$ 
and $a_1,\ldots,a_n \in \Z$, then
\[
\chi_j\left (\sum_{i=1}^n a_i \, g_i\right) =\sum_{i=1}^n a_i \,\chi_j(g_i) 
\in \Z[\zeta_{q_j}] .
\]
This shows that the restriction of $\phi$ to $\Z G$ 
defines an algebra homomorphism from $\Z G$ into 
$\Z[\zeta_{q_1}]   \oplus \cdots   \oplus \Z[\zeta_{q_k}]$. 
The other claims follows from Proposition \ref{th-qg}.
\end{proof}
We remark that the homomorphism $\psi$ is not surjective.

\section{Ideals in products of integral domains}\label{sec-restr}

In the next section we are going to exploit Theorem \ref{th-nibg} and 
Proposition \ref{th-zg} in order to give an explicit algorithm to find 
a normal integral basis of $L$. 
Here we state some other preliminary results.

Since it is quite difficult to work in $\Z[G]$, we will consider 
the extension of the ideal $I$ by means of the homomorphism $\psi$. 
Then we would like to recover a generator of $I$ starting from 
a generator of such an extended ideal, whenever they exist. 

We first state a simple lemma about product of rings.  
If $\Gamma$ is any ring, we denote its group of units by $U(\Gamma)$.

\begin{lemma}\label{th-rproduct}
Let $\Gamma$ be commutative ring which can be expressed as a direct product
$\Gamma = \Gamma_1 \times \cdots \times \Gamma_k $
of some rings $\Gamma_i$, where $i=1,\ldots,k$.
\begin{itemize}
\item
Any ideal $J$ of $\Gamma$ has the form $J = J_1 \times \cdots \times J_k$, 
where $J_i$ is an ideal of $\Gamma_i$.
\item
An ideal $J$ of $\Gamma$ is principal if and only if $J_i$ is principal, 
for all $i$; in such a case, if $ J = a \Gamma$ with 
$a = (a_1,\ldots,a_k) \in \Gamma$, then $J_i = a_i\Gamma_i$.
\item
If each $\Gamma_i$ is an integral domain and $a,b \in \Gamma$, 
then $a\Gamma=b \Gamma$ if and only if there is a $u \in U(\Gamma)$
such that $a = ub$.
\item
$u \in U(\Gamma)$ if and only if $u=(u_1,\ldots,u_k)$, with $u_i \in U(\Gamma_i)$.
\end{itemize}
\end{lemma}

Now we let $\Gamma= \Z[\zeta_{q_1}] \times \cdots \times \Z[\zeta_{q_k}]$ 
and $J = \psi(I) \Gamma$, where $I$ is defined by (\ref{eq-ideal}), and
apply the previous lemma to them. 
The next result explains how we can recover a generator of $I$, 
when it exists. 

\begin{proposition}\label{th-generators}
Let $I$ be a principal ideal of $\Z[G]$, $I=t  \Z[G]$, and let $J=\psi(I) \Gamma$. 
Then $J$ is principal, and if $J=d \Gamma $, then $t = \psi^{-1}(ud)$, for some 
$u \in U(\Gamma)$. 
\end{proposition}
\begin{proof}
It is obvious that $J$ is principal, since $J =\psi (t)  \Gamma$. 
If $ J = \Gamma d $, Lemma \ref{th-rproduct} implies that $\psi(t) = ud $, 
for some $u \in U(\Gamma)$, whence the result.
\end{proof}
The proof of the following proposition is easy.
\begin{proposition}\label{aggiunto}
Let $S$ be a set of coset representatives of $U(\Z G)$ in   $U(\Gamma)$.
Let $I$ a principal ideal of $\Z[G]$, $I=t  \Z[G]$, and let $J=\psi(I) \Gamma$. 
Then $J$ is principal, and if $J=d  \Gamma$, then $t = \psi^{-1}(ud)$, for some 
$u \in S$. 
\end{proposition}

\section{Outline of the algorithm}\label{sec-algo}
Let us assume that $L$ is an abelian number field of degree $n$ 
with Galois group $G$, 
and that $L$ it is described by giving the 
minimal polynomial $m(x)$ of an integral primitive element $\alpha$, 
which generates a normal basis for $L/\Q$. 
We assume that the Galois group $G$ has been computed, that we know 
the action of its elements on $\alpha$, and that we have fixed 
an ordering of them, say $\left( g_1,\ldots,g_n \right)$. 
We assume also that an integral basis of $\OL$ has 
been computed, and that we have fixed an ordering of its elements, 
say $\left( \beta_1,\ldots,\beta_n \right)$.

We first compute the discriminant $D$ of the conjugates of $\alpha$,  
and determine $\alpha': = \alpha/D$. 
We then compute the matrix $(b_{ij})$ defined by (\ref{eq-bij}).

Then we fix an ordering of the irreducible characters $\chi_i$ of $G$, 
take note of their orders $q_i$,  
and compute the character matrix $A:=(\chi_i(g_j))$.
We reorder the characters and extract the matrix $B$ from $A$, 
as is described in Proposition \ref{th-qg}.

Next, we compute the ideal $J:=\psi(I)R$, which is generated over $R$ 
by $\psi(b_1)$, $\ldots$, $\psi(b_n)$. 
This is done by applying (\ref{eq-ci-b}). 
Namely, for $j=1,\ldots,n$, if 
\begin{equation}\label{algo-bij}
b_j = \sum_{i=1}^n b_{ij} \, g_i ,
\end{equation}
then $\psi(b_j) = (c_{1,j},\ldots,c_{k,j})$, where
\begin{equation}\label{algo-cij}
\left( \begin{array}{c} c_{1, j} \\ \vdots \\ c_{k, j} \end{array} \right) 
= B 
\left( \begin{array}{c} b_{1,j} \\ \vdots \\ b_{n,j} \end{array} \right) . 
\end{equation}

Then, for each $i= 1, \ldots, k$, we let $J_i$ be the ideal 
of $\Z[\zeta_{q_i}]$ generated by the set 
$ \left\{ c_{i, 1} , \ldots, c_{i, n} \right\} $ - in other words, the ideal $J_i$ is generated 
by the $i$-th row of the matrix $(b_{ij}) \cdot B$.

Next we must find, for each $i = 1, \ldots, k $, 
a generator $d_i$ of the ideal $J_i$, whenever it exists.
For this purpose we use the Sage function   {\em is\_principal}
which returns True if the ideal    $J_i$     is principal,
followed by the function {\em gens\_reduced}
which expresses $J_i$     in terms of at most two generators, 
and one if possible.

Now, note that if an irreducible character $\chi_i$ maps $I$ into $J_i$ 
and $\sigma$ is an automorphism of $\Q(\zeta_{q_i})$, where $i=1,\dots,k$, 
then: 
\begin{itemize}
	\item $\sigma(\chi_i)$ maps $I$ into $\sigma(J_i)$;
	\item if $d_i$ is a generator of $J_i$ then 
$\sigma(d_i)$ is a generator of $\sigma(J_i)$.
\end{itemize}
Therefore the remaining elements $d_{k+1},\ldots, d_n$ are obtained 
by applying to each $d_i$ all the non--trivial automorphisms of 
$\Q(\zeta_{q_i})$, where $i=1,\ldots,k$. 

Now, Lemma \ref{th-rproduct} tells us that the element 
$d:=(d_1, \ldots, d_n)$ generates $J$. 
We put $t:=\psi^{-1}(d)$. 
In order to recover the standard form of $t$ we just apply 
(\ref{eq-hi-b}). Hence, 
\begin{equation}\label{algo-t}
t = \sum_{i=1}^n t_i g_i , 
\end{equation}
where
\begin{equation}\label{algo-ti}
\left( \begin{array}{c} t_1 \\ \vdots \\ t_n \end{array} \right) 
 = A^{-1} 
\left( \begin{array}{c} d_1 \\ \vdots \\ d_n \end{array} \right) . 
\end{equation}

Let us point out that all these computations are done in $\Q(\zeta_r)$,
where $r$ is the exponent of $G$.
Now we let $\theta:=t\alpha'$, and we compute the discriminant of  
$\{g_1\theta,\ldots,g_n\theta\}$.

If $\mathrm{D}(g_1\theta,\ldots,g_n\theta)$ 
equals the discriminant of $L$ {\em   and }   $\theta$ 
is an algebraic integer then $\theta$ generates a normal integral basis. 

Otherwise, we   \lq adjust\rq\ $d$ by multiplying it by a suitable unit $u \in S$, 
and then we repeat the last matrix multiplication.
Precisely, we first multiply the element $(d_1, \ldots, d_k)$ 
by a unit $u=(u_1, \ldots, u_k) \in S$,
then we obtain the remaining elements $d_{k+1},\ldots, d_n$ by 
applying to each $d_i$ all the non--trivial automorphisms of 
$\Q(\zeta_{q_i})$, where $i=1,\ldots,k$,
and finally we apply again formula (\ref{algo-ti}). In this way, 
we obtain a new $\theta$  and we have 
now to verify whether $\mathrm{D}(g_1\theta,\ldots,g_n\theta)$ 
equals the discriminant of $L$ {\em   and }  it is an algebraic integer or not. 

The process must terminate after a finite number of steps, since the cardinality of $S$, i.e. 
the index $(U(\Gamma):U(\Z G))$
is finite. 
Furthermore, Proposition \ref{th-generators} implies that we 
will surely find a normal basis generator in one of these steps, 
if it exists.

\section{Computation of the set $S$}
Our task is to compute efficiently a set $S$ of coset representatives of $U(\Z G)$ in $U(\Gamma)$.

This task could be accomplished, for example, by first computing
$U(\Gamma)$  using standard algorithms in algebraic number theory,  
 then   computing $U(\Z G)$  by means of the  algorithm 
 designed by P. Faccin \cite[2.5]{faccin}, and, finally,  computing the sought set of
 coset representatives.
However, the computational effort involved by this method is not justified in this context.

So, we have to revert to a different method.
E. Teske \cite{teske1} and, later, J. Buchmann and A. Schmidt \cite{buchmann} 
developed some efficient algorithms  allowing one
to compute the structure of an unknown finite abelian group $X$,
i.e. its invariant factors,
assuming that a    set $V$ of generators is given, and that
for $a, b \in X$ it is possible to compute their  product,
it is possible to compute $a^{-1}$, and it is possible to test the equality $a=b$.
Buchmann's   algorithm requires  $O( |V|  \sqrt { |X|})$ group operations
and stores $O(\sqrt{|X|})$ group elements.
Once we have expressed $X$ as a direct product of cyclic invariant factors
we can list easily all its elements, in $O(|X|)$ time.

In our case $X=U(\Gamma)/U(\Z G)$.
Now, we know    a set   $V$ of generators of $U(\Gamma)$ 
(they can be computed in Sage for each $\Gamma_i$ using the function  {\em UnitGroup}).
We define the product of two elements $a, b \in U(\Gamma)$ as their ordinary product in $\Gamma$.
We define the inverse of an element $a \in U(\Gamma)$ as its ordinary inverse in $\Gamma$.
Finally, we decreed two elements $a$ and $b$ to be equal
whenever they are congruent modulo $U(\Z G)$. To test if two elements of $U(\Gamma)$ are congruent modulo $U(\Z G)$, 
we must divide $a$ by $b$, and the result must be in $U(\Z G)$; but this is equivalent to say that $a / b$  lies in $ \Z G$, since both $a$ and $b$ are units of $\Gamma$, and thus $a/b$ is a unit of $\Gamma$ lying in $\Z G$.

As far as it concerns $|X|$, we have the following estimate proved in   \cite{io}:
\begin{theorem}\label{main}
Let $G$ be a finite abelian group of order $n$. If $a_q$
stands for the number of cyclic subgroups of order $q$ of $G$, then:
\begin{equation}
 ( U(\Gamma) : U(\Z G) )   <  		 n^n 
		 	\left(   \Pi_{q|n}
			{     
				\left(
				 \frac{   q^{\phi(q)}  }
					{  \Pi_{p|q} p^{\phi(q) /(p-1)}   }  
				\right)^{a_q}
			}    \right)^{-1}         
\end{equation}
\end{theorem}

\section{Some worked out examples}\label{sec-exmpl}
We computed the following examples with the public domain computer algebra system 
SAGE \cite{sage}. Unfortunately,  the current version of Sage does not allow one
to compute Galois groups of fields of degree greater than 11 natively,
i.e. without installing extra  packages.
\begin{example}
Consider the polynomial
$x^{12} + 8*x^{11} - 837*x^{10} - 98016*x^9 - 9093374*x^8 + 971323080*x^7 + 88039800038*x^6 + 3042444275430*x^5 +  
67073014243125*x^4 - 3252703653719588*x^3 - 94326521098073965*x^2 + 3079043710339656342*x + 75641678543561531059$. 
The discriminant of its splitting field $L$ is $205924456521$.
A generator of a normal integral basis   exists in this case, and its minimal polynomial is 
$x^{12} - x^{11} + x^9 - x^8 + x^6 - x^4 + x^3 - x + 1$,
which 'resembles'  the 13-th cyclotomic polynomial  
$x^{12} + x^{11} + x^{10} + x^9 + x^8 + x^7 + x^6 + x^5 + x^4 + x^3 + x^2 + x +1$.
\end{example}
\begin{example}
Consider the polynomial 
$x^8 + 20*x^7 + 800*x^6 + 12485*x^5 + 235045*x^4 + 2387800*x^3 + 24032600*x^2 - 34407800*x + 62712400$.
The discriminant of its splitting field $L$ is 1265625. 
A generator of a normal integral basis   exists in this case, and its minimal polynomial is 
 $x^8 + x^7 - x^5 - x^4 - x^3 + x + 1$, which 'resembles'  the 15-th
 cyclotomic polynomial $x^8 - x^7 + x^5 - x^4 + x^3 - x + 1$.
\end{example}

\end{document}